\documentclass[twoside,leqno]{article}
\usepackage{amsfonts}
\usepackage{amsmath}
\usepackage{amssymb}
\usepackage{amsthm}
\usepackage{graphicx}
\usepackage{epsfig}
\usepackage{amsmath,amssymb,amsfonts}
\usepackage{lineno}

\def\bjga#1{\def\thefootnote{}\footnote{
    \hspace*{-.54cm} $~~~~~$The author is thankful to the National Board for Higher Mathematics (NBHM), India, for their financial support(Ref No: 0203/11/2017/RD-II/10440) to carry on this research work.}
    \addtocounter{footnote}{-1}\def\thefootnote{\arabic{footnote}}}
\def\mymaketitle#1{\date{\bjga{#1}}\begin{document}\maketitle\thispagestyle{empty}}

\def\babs{\begin{abstractv}}
\def\eabs{\end{abstractv}}

\theoremstyle{definition}
\newtheorem{definition}{Definition}[section]
\newtheorem{example}[definition]{Example}
\newtheorem{remark}[definition]{Remark}
\newtheorem{problem}[definition]{Problem}

\theoremstyle{theorem}
\newtheorem{theorem}{Theorem}[section]
\newtheorem{proposition}[theorem]{Proposition}
\newtheorem{corollary}[theorem]{Corollary}
\newtheorem{lemma}[theorem]{Lemma}

\def\pas{\par\smallskip}
\def\pasn{\par\smallskip\noindent}
\def\pam{\par\medskip}
\def\pamn{\par\medskip\noindent}
\def\pab{\par\bigskip}

\def\header#1#2#3#4#5{
    \markboth{#3}{#4}
    \title{#5}
    \author{#3}
    \date{}
    \mymaketitle{#1-#2}}
\newenvironment{abstractv}{\begin{quote}{\bf Abstract.\ }}{\end{quote}}
\def\msc{{\bf M.S.C. 2010}:\ }
\def\kwd{\\{\bf Key words}:\ }
\def\aua{\par\noindent{\em Author's address:}\pam\noindent}
\def\auas{\par\noindent{\em Authors' addresses:}\pam\noindent}
\def\auac{\par\noindent{\em Authors' address:}\pam\noindent}
\def\bece{\begin{center}}
\def\eece{\end{center}}
\def\bebi{\begin{thebibliography}{40}\setlength{\parskip}{-1mm}}
\def\eebi{\end{thebibliography}}
\def\bibi#1{\bibitem{#1}}

\newcommand{\R}{\mathbb{R}}
\newcommand{\Rt}{\mbox{{\em $\R$}}}
\newcommand{\Rs}{\mbox{\tiny{\R}}}
\newcommand{\C}{\mathbb{C}}
\newcommand{\Cs}{\mbox{\tiny{\C}}}
\newcommand{\Q}{\mathbb{Q}}
\newcommand{\Z}{\mathbb{Z}}
\newcommand{\Zs}{\mbox{\tiny{\Z}}}
\newcommand{\N}{\mathbb{N}}
\def\Ne{\mbox{\em $\N$}}
\newcommand{\Ns}{\mbox{\tiny{\N}}}
\newcommand{\Cbar}{\bar{\mathbf{C}}}
\def\DD{\;\mbox{D}\!\!\!\!\!\!\mbox{I}\;\;\,}
\def\DDs{\mbox{{\scriptsize$\;\mbox{D}\!\!\!\!\!\!\mbox{I}\;\;\,$}}}

\def\noi{\noindent}
\def\qq{\qquad}
\def\mm{\medskip\\}
\def\lg{\langle}
\def\rg{\rangle}
\def\ra{\Righarrow}
\def\lra{\Leftrightarrow}
\def\ri{\rightarrow}
\def\fall{\mbox{ for all }}
\def\di{\displaystyle}
\def\vp{\varphi}
\def\al{\alpha}
\def\ol#1{\overline{#1}}
\def\qed{\hfill$\Box$}
\def\bref#1{(\ref{#1})}
\def\midd{\hspace*{-10pt}\left.\phantom{\di\int}\right|}
\def\text#1{\mbox{#1}}
\def\emph#1{{\em #1}}
\def\textit#1{{\em #1}}
\def\textbf#1{{\bf #1}}
\def\func#1{\mathop{\rm #1}}
\def\limfunc#1{\mathop{\rm #1}}
\def\dint{\displaystyle\int}
\def\dsum{\displaystyle\sum}
\def\dfrac{\displaystyle\frac}
\def\Bbb#1{\mathbb{#1}}
\def\bu{$\bullet$\ }
\def\sta{$\star$\ }
\def\ii{\,\mbox{i}\,}
\def\lb{\linebreak}
\def\pr{{}^{\prime}}
\def\imath{\mbox{i}}

\def\ba{\begin{array}}
\def\ea{\end{array}}
\def\beq{\begin{equation}}
\def\eeq{\end{equation}}
\def\zx#1{\begin{equation}\label{#1}}
\def\zc{\end{equation}}
\def\aru#1{\left\{{\begin{array}{l}#1\end{array}}\right.}
\def\matd#1{\left(\begin{array}{cc}#1\end{array}\right)}

\def\shw{\scriptscriptstyle}
\def\stkdn#1#2{{\mathop{#2}\limits^{}_{#1}}{}}
\def\dn#1#2{\stkdn{{\shw #1}}{#2}{}}
\def\stkup#1#2{{\mathop{#1}\limits^{#2}_{}}{}}
\def\up#1#2{\stkup{#1}{\shw #2}{}}
\def\stkud#1#2#3{{\mathop{#2}\limits^{#3}_{#1}}{}}
\def\ud#1#2#3{\stkud{\shw #1}{#2}{\shw #3}}
\def\stackunder#1#2{\mathrel{\mathop{#2}\limits_{#1}}}

\def\emptyset{\begin{picture}(13,10) \unitlength1pt
    \put(5,3){\circle{7}}\put(.7,-1.8){\line(1,1){10}}\end{picture}}

\def\fiu#1#2#3{\begin{center}
    \includegraphics[scale=#1]{#2.eps}\nopagebreak\\\parbox{10cm}
    {\begin{center}\small\bf #3\end{center}}\end{center} }
\def\fid#1#2#3#4{\begin{center}
    \includegraphics[scale=#1]{#2.eps}\hspace*{.7cm}
    \includegraphics[scale=#1]{#3.eps}
    \nopagebreak\\\parbox{14cm}{\begin{center}
    \small\bf #4\end{center}}\end{center} }
\def\fit#1#2#3#4#5{\begin{center}\begin{tabular}{ccc}
    \includegraphics[scale=#1]{#2.eps} &\hspace*{.4cm}&
    \includegraphics[scale=#1]{#3.eps}\\
    {\small\bf #4}&&{\small\bf #5}\end{tabular}\end{center}}
\pagestyle{myheadings}
\setlength{\textheight}{20cm}
\setlength{\textwidth}{13cm}
\setlength{\oddsidemargin}{18mm}
\setlength{\evensidemargin}{18mm}
\renewcommand{\theequation}{\thesection.\arabic{equation}}
\makeatletter \@addtoreset{equation}{section} \makeatother

\header{30}{36}{Dipen Ganguly}
    {Kenmotsu metric as conformal $\eta$-Ricci soliton}
    {Kenmotsu metric as conformal $\eta$-Ricci soliton}
\babs
    The object of the present paper is to characterize the class of Kenmotsu manifolds which admits conformal $\eta$-Ricci soliton. Here, we have investigated the nature of the conformal $\eta$-Ricci soliton within the framework of Kenmotsu manifolds. It is shown that an $\eta$-Einstein Kenmotsu manifold admitting conformal $\eta$-Ricci soliton is an Einstein one. Moving further, we have considered gradient conformal $\eta$-Ricci soliton on Kenmotsu manifold and established a relation between the potential vector field and the Reeb vector field. Next, it is proved that under certain condition, a conformal $\eta$-Ricci soliton on Kenmotu manifolds under generalized D-conformal deformation remains invariant. Finally, we have constructed an example for the existence of conformal $\eta$-Ricci soliton on Kenmotsu manifold.
\eabs
\msc
    53C15, 53C25, 53C44.
\kwd
    Ricci flow; Ricci soliton; Conformal $\eta$-Ricci soliton; \\
    Kenmotsu manifold; Generalized $D$-conformal deformation

\section{Introduction}
$~~~$A smooth manifold $M$ equipped with a Riemannian metric $g$ is said to be a Ricci soliton, if for some constant $\lambda$, there exist a smooth vector field $V$ on $M$ satisfying the equation
\begin{equation}
Ric+\frac{1}{2}\mathcal{L}_{V}g=\lambda g,
\end{equation}
where $\mathcal{L}_{V}$ denotes the Lie derivative along the direction of the vector field $V$ and $Ric$ is the Ricci tensor. The Ricci soliton is called shrinking if $\lambda >0$, steady if $\lambda =0$ and expanding if $\lambda <0$. In 1982, R. S. Hamilton \cite{H1} initiated the study of Ricci flow as a self similar solution to the Ricci flow equation given by
\begin{equation}
\frac{\partial g}{\partial t}=-2Ric.\nonumber
\end{equation}
\par
Ricci soliton also can be viewed as natural generalization of Einstein metric which moves only by an one-parameter group of diffeomorphisms and scaling \cite{H2}. After Hamilton, the significant work on Ricci flow has been done by G. Perelman to prove the well known Thurston's geometrization conjecture. \par
\medskip
A. E. Fischer \cite{Fis} in 2005, introduced conformal Ricci flow equation which is a modified version of the Hamilton's Ricci flow equation that modifies the volume constraint of that equation to a scalar curvature constraint. The conformal Ricci flow equations on a smooth closed connected oriented n-manifold, $n\geq 3$, are given by
\begin{equation}
\frac{\partial g}{\partial t}+2(Ric+\frac{g}{n})=-pg,~~~~r(g)=-1,\nonumber
\end{equation}
where $p$ is a non-dynamical(time dependent) scalar field and $r(g)$ is the scalar curvature of the manifold. The term $-pg$ acts as the constraint force to maintain the scalar curvature constraint in the above equation. Note that these evolution equations are analogous to famous Navier-Stokes equations where the constraint is divergence free. The non-dynamical scalar $p$ is also called the conformal pressure.
\par
\medskip
Later in 2015, N. Basu and A. Bhattacharyya \cite{NBa} introduced the concept of conformal Ricci soliton as a generalization of the classical Ricci soliton and is given by the equation
\begin{equation}
\mathcal{L}_{V}g+2Ric=[2\lambda -(p+\frac{2}{n})]g,
\end{equation}
where $\lambda$ is a constant and $p$ is the conformal pressure. It is to be noted that the conformal Ricci soliton is a self-similar solution of the Fisher's conformal Ricci flow equation. After that a lot of work have been done on conformal Ricci solitons on various geometric structures like $(LCS)_n$-manifolds \cite{DG2} and generalized Sasakian space forms \cite{DG3}.
\par
\medskip
\medskip
Again a Ricci soliton is called a gradient Ricci soliton if the potential vector field $V$ in the equation (1.1) is the gradient of some smooth function $f$ on $M$. This function $f$ is called the potential function of the Ricci soliton. J. T. Cho and M. Kimura \cite{ChoK} introduced the concept of $\eta$-Ricci soliton and later C. Calin and M. Crasmareanu \cite{CaCra} studied it on Hopf hypersufaces in complex space forms. A Riemannian manifold $(M,g)$ is said to admit an $\eta$-Ricci soliton if for a smooth vector field $V$, the metric $g$ satisfies the following equation
\begin{equation}
\mathcal{L}_{V}g+2S+2\lambda g+2\mu \eta \otimes \eta=0,
\end{equation}
where $\mathcal{L}_{V}$ is the Lie derivative along the direction of $V$, $S$ is the Ricci tensor and $\lambda$, $\mu$ are real constants. It is to be noted that for $\mu=0$ the $\eta$-Ricci soliton becomes a Ricci soliton.\par
\medskip
Combining equations $(1.2)$ and $(1.3)$ M. D. Siddiqi \cite{MDS1} introduced the notion of conformal $\eta$-Ricci soliton given by the following equation
\begin{equation}
\mathcal{L}_{V}g+2S+[2\lambda -(p+\frac{2}{n})]g+2\mu \eta \otimes \eta=0,
\end{equation}
where $\mathcal{L}_{V}$ is the Lie derivative along the direction of $V$, $S$ is the Ricci tensor, $n$ is the dimension of the manifold, $p$ is the non-dynamical scalar field (conformal pressure) and $\lambda$, $\mu$ are real constants. In particular if $\mu=0$ the conformal $\eta$-Ricci soliton reduces to the conformal Ricci soliton.\par
\medskip
Furthermore, if the vector field $V$ is the gradient of some smooth function $f$, then the data $(g,V,\lambda,\mu)$ is called a gradient conformal $\eta$-Ricci soliton. Thus if $V=Df$, where $Df=gradf$, then the equation $(1.4)$ becomes
\begin{equation}
Hessf+S+[\lambda -(\frac{p}{2}+\frac{1}{n})]g+\mu \eta \otimes \eta=0,
\end{equation}
where $Hessf$ denotes the Hessian of the smooth function $f$. In this case the vector field $V$ is called the potential vector field and the smooth function $f$ is called the potential function.
\par
\medskip
On the other hand, contact geometry plays a pivotal role in the study of modern differential geometry. K. Kenmotsu \cite{KK} introduced a special class of contact Riemannian manifolds, satisfying certain conditions, which was later named as Kenmotsu manifold. Kenmotsu proved that a locally Kenmotsu manifold is a warped product $I\times_fN$ of an interval $I$ and a K\"{a}hler manifold $N$ with warping function $f(t)=ke^t$, where $k$ is a non-zero constant. Over the years many authors have studied the geometry of Kenmotsu manifolds \cite{DGPHui}. A. Ghosh \cite{AG} first investigated Kenmotsu $3$-metric as Ricci solitons. Motivated by the above studies, in this paper, we have considered conformal $\eta$-Ricci solitons on Kenmotsu manifolds.
\par
\medskip
The outline of the article goes as follows: After a brief introduction, in section $2$ we recall some basic results of Kenmotsu manifolds. Section $3$ deals with Kenmotsu manifolds admitting conformal $\eta$-Ricci solitons. In section $4$, we consider gradient conformal $\eta$-Ricci solitons on Kenmotsu manifolds. Section $5$ is devoted to the study of conformal $\eta$-Ricci solitons on Kenmotsu manifolds under generalized $D$-conformal deformation. In section $6$, we have studied torse-forming vector field on $\epsilon$-Kenmotsu manifolds admitting conformal $\eta$-Ricci solitons. Lastly, we have constructed an example to illustrate the existence of conformal $\eta$-Ricci soliton on Kenmotsu manifold.
\medskip
\medskip

\section{Preliminaries}
$~~~$Let $(M,g)$ be a $(2n+1)$-dimensional smooth Riemannian manifold that admits a $(1,1)$ tensor field $\phi$, a characteristic vector field $\xi$, a global 1-form $\eta$ on M satisfying the following relations
\begin{eqnarray}
  \phi^2 &=& -I+\eta\otimes\xi,~~~\eta(\xi)=1, \\
  \eta(X) &=& g(X,\xi), \\
  g(\phi X,\phi Y) &=& g(X,Y)-\eta(X)\eta(Y)
\end{eqnarray}
for all vector fields $X,Y\in TM$, where $TM$ is the tangent bundle of the manifold $M$. Then we say that $(M,g,\phi,\xi,\eta)$ is an almost contact metric manifold \cite{Bla}. Note that from the above relations $\phi(\xi)=0$, $\eta(\phi X)=0$ $\forall X\in TM$.\par
\medskip
 If the Levi-Civita connection $\nabla$ of an almost contact metric manifold $(M,g,\phi,\xi,\eta)$ satisfies
\begin{equation}
  (\nabla_X\phi)Y=g(\phi X,Y)-\eta(Y)\phi X,
\end{equation}
 $\forall X,Y\in TM$, then the manifold $(M,g,\phi,\xi,\eta)$ is said to be a Kenmotsu manifold \cite{KK}. Observe that for all smooth vector fields $X$ on a Kenmotsu manifold $M$, from $(2.4)$ we have
\begin{equation}
  \nabla_X\xi=X-\eta(X)\xi.
\end{equation}\par
\medskip
Furthermore, in a Kenmotsu manifold the following relations hold \cite{KK}
\begin{eqnarray}
  (\nabla_X\eta)Y &=& g(X,Y)-\eta(X)\eta(Y), \\
  R(X,Y)\xi &=& \eta(X)Y-\eta(Y)X, \\
  R(\xi,X)Y &=& \eta(Y)X-g(X,Y)\xi, \\
  \eta(R(X,Y)Z) &=& g(X,Z)\eta(Y)-g(Y,Z)\eta(X), \\
  S(X,\xi) &=& -2n\eta(X),
\end{eqnarray}
where $R$ is the curvature tensor, $S$ is the Ricci tensor and $Q$ is the Ricci operator given by $g(QX,Y)=S(X,Y)$, for all $X,Y\in TM$.\par
\medskip
\begin{definition}
  A Kenmotsu manifold $(M,g,\phi,\xi,\eta)$ is said to be an $\eta$-Einstein manifold if its Ricci tensor $S$ satisfies
  \begin{equation}
  S(X,Y)=\alpha g(X,Y)+\beta \eta(X)\eta(Y),
  \end{equation}
  for all $X,Y\in TM$ and smooth functions $\alpha,\beta$ on $M$. In particular, if $\beta=0$, then we say that the manifold $M$ is an Einstein manifold.
\end{definition}
\medskip
\medskip
\section{Kenmotsu manifolds admitting conformal $\eta$-Ricci solitons}
$~~~$In this section we consider a Kenmotsu manifolds admitting conformal $\eta$-Ricci solitons and first we try to characterize the nature of the soliton and then we study the effect of the soliton on the geometry of the underlying manifold.
\medskip
\par
 Let us first state the following lemma (for details see \cite{Venk}) which will be used later in this section.

\begin{lemma}
  Let $(M,g,\phi,\xi,\eta)$ be a $(2n+1)$-dimensional Kenmotsu manifold. Then the Ricci operator $Q$ satisfies the following relations
  \begin{eqnarray}
    (\nabla_XQ)\xi &=& -QX-2nX \\
    (\nabla_\xi Q)X &=& -2QX-4nX
  \end{eqnarray}
  for all smooth vector field $X\in TM$.
\end{lemma}

Now we are going to find the conditions under which a conformal $\eta$-Ricci soliton is shrinking, steady or expanding on a Kenmotsu manifold.

\begin{theorem}
  A conformal $\eta$-Ricci soliton $(g,V,\lambda,\mu)$, on a $(2n+1)$-dimensional Kenmotsu manifold $(M,g,\phi,\xi,\eta)$, is shrinking if $\mu <2n+(\frac{p}{2}+\frac{1}{2n+1})$, steady if $\mu =2n+(\frac{p}{2}+\frac{1}{2n+1})$ and expanding if $\mu >2n+(\frac{p}{2}+\frac{1}{2n+1})$.
\end{theorem}

\begin{proof}
Since the data $(g,V,\lambda,\mu)$ is a conformal $\eta$-Ricci soliton, from $(1.4)$ it follows that
\begin{equation}
(\mathcal{L}_{V}g)(X,Y)+2S(X,Y)+[2\lambda -(p+\frac{2}{2n+1})]g(X,Y)+2\mu \eta(X)\eta(Y)=0,
\end{equation}
 for all smooth vector fields $X,Y\in TM$.\\
 Taking covariant differentiation of $(3.3)$ along an arbitrary vector field $W$ we get
 \begin{eqnarray}
(\nabla_W\mathcal{L}_{V}g)(X,Y) &=& -2(\nabla_WS)(X,Y)-2\mu[g(W,X)\eta(Y) \nonumber\\
&& +g(W,Y)\eta(X)-2\eta(W)\eta(X)\eta(Y)].
\end{eqnarray}
Let us now recall the well known commutation formula by Yano (See \cite{Ya})
\begin{equation}
(\mathcal{L}_{V}\nabla_Xg-\nabla_X\mathcal{L}_{V}g-\nabla_{[V,X]}g)(Y,Z)=-g((\mathcal{L}_{V}\nabla)(X,Y),Z)-g((\mathcal{L}_{V}\nabla)(X,Z),Y),
\end{equation}
for all smooth vector fields $X,Y,Z\in TM$.\\
As $g$ is a metric connection, $\nabla g=0$, and using this in $(3.5)$ yields
\begin{equation}
(\nabla_X\mathcal{L}_{V}g)(Y,Z)=g((\mathcal{L}_{V}\nabla)(X,Y),Z)+g((\mathcal{L}_{V}\nabla)(X,Z),Y).\nonumber
\end{equation}
Since $\mathcal{L}_{V}\nabla$ is symmetric, the previous equation reduces to
\begin{equation}
g((\mathcal{L}_{V}\nabla)(X,Y),Z)=\frac{1}{2}[(\nabla_X\mathcal{L}_{V}g)(Y,Z)+(\nabla_Y\mathcal{L}_{V}g)(Z,X)-(\nabla_Z\mathcal{L}_{V}g)(X,Y)].
\end{equation}
Recalling $(3.4)$ and using it in $(3.6)$ we obtain
\begin{eqnarray}
g((\mathcal{L}_{V}\nabla)(X,Y),Z) &=& (\nabla_ZS)(X,Y)-(\nabla_XS)(Y,Z)-(\nabla_YS)(Z,X) \nonumber\\
&& -2\mu[g(X,Y)\eta(Z)-\eta(X)\eta(Y)\eta(Z)].
\end{eqnarray}
Putting $Y=\xi$ in $(3.7)$ we get
\begin{equation}
g((\mathcal{L}_{V}\nabla)(X,\xi),Z)=(\nabla_ZS)(X,\xi)-(\nabla_XS)(\xi,Z)-(\nabla_\xi S)(Z,X).
\end{equation}
Again we know that
\begin{equation}
(\nabla_ZS)(X,Y)=ZS(X,Y)-S(\nabla_ZX,Y)-S(X,\nabla_ZY),\nonumber
\end{equation}
for all $X,Y,Z\in TM$. Then using $S(X,Y)=g(QX,Y)$ and $(3.1), (3.2)$ in the foregoing equation we can easily deduce
\begin{eqnarray}
  (\nabla_ZS)(X,\xi) &=& -S(X,Z)-2ng(X,Z) \\
  (\nabla_\xi S)(Z,X) &=& -2S(X,Z)-4ng(X,Z).
\end{eqnarray}
Making use of $(3.9)$ and $(3.10)$ in the equation $(3.8)$ we obtain
\begin{equation}
g((\mathcal{L}_{V}\nabla)(X,\xi),Z)=2S(X,Z)+4ng(X,Z),\nonumber
\end{equation}
which essentially implies
\begin{equation}
(\mathcal{L}_{V}\nabla)(X,\xi)=2QX+4nX.
\end{equation}
Differentiating $(3.11)$ along the direction of an arbitrary vector field $Z$ and using $(2.5)$ we arrive at
\begin{equation}
(\nabla_Z\mathcal{L}_{V}\nabla)(X,\xi)=-(\mathcal{L}_{V}\nabla)(X,Z)+2\eta(Z)[QX+2nX]+2(\nabla_ZQ)X.
\end{equation}
Now, from Yano \cite{Ya} we recall the following relation
\begin{equation}
(\mathcal{L}_{V}R)(X,Y)Z=(\nabla_X\mathcal{L}_{V}\nabla)(Y,Z)-(\nabla_Y\mathcal{L}_{V}\nabla)(X,Z).
\end{equation}
Setting $Z=\xi$ in the foregoing equation and then making use of $(3.12)$ and using the symmetry of $(\mathcal{L}_{V}\nabla)$ we get
\begin{eqnarray}
(\mathcal{L}_{V}R)(X,Y)\xi &=& 2[(\nabla_XQ)Y-(\nabla_YQ)X+\eta(X)QY-\eta(Y)QX]\nonumber\\
&&+4n[\eta(X)Y-\eta(Y)X].
\end{eqnarray}
Taking $Y=\xi$ in the preceding equation and recalling $(3.1), (3.2)$ we obtain
\begin{equation}
(\mathcal{L}_{V}R)(X,\xi)\xi=0.
\end{equation}
Again setting $Y=\xi$ in $(3.3)$ and using $(2.10)$ we get
\begin{equation}
(\mathcal{L}_{V}g)(X,\xi)=[4n-2\lambda-2\mu+(p+\frac{2}{2n+1})]\eta(X).
\end{equation}
Recalling equation $(2.2)$, i.e.; $g(X,\xi)=\eta(X)$ and taking its Lie derivative along $V$ gives
\begin{equation}
(\mathcal{L}_{V}g)(X,\xi)=(\mathcal{L}_{V}\eta)(X)-g(X,\mathcal{L}_{V}\xi).
\end{equation}
Substituting $(3.16)$ in the previous equation we obtain
\begin{equation}
(\mathcal{L}_{V}\eta)(X)-g(X,\mathcal{L}_{V}\xi)=[4n-2\lambda-2\mu+(p+\frac{2}{2n+1})]\eta(X).
\end{equation}
Taking $X=\xi$ in $(3.18)$ gives
\begin{equation}
\eta(\mathcal{L}_{V}\xi)=[\lambda+\mu-2n-(\frac{p}{2}+\frac{1}{2n+1})].
\end{equation}
Again from $(2.5)$ we can write
\begin{equation}
R(X,\xi)\xi=\eta(X)\xi-X.
\end{equation}
Lie differentiating $(3.20)$ along an an arbitrary vector field $V$ we get
\begin{equation}
\mathcal{L}_{V}R(X,\xi)\xi=(\mathcal{L}_{V}\eta(X))\xi+\eta(X)\mathcal{L}_{V}\xi-\mathcal{L}_{V}X.
\end{equation}
Also from $(2.8)$ we can easily compute that
\begin{equation}
R(X,\xi)\mathcal{L}_{V}\xi=-\eta(\mathcal{L}_{V}\xi)X+g(X,\mathcal{L}_{V}\xi)\xi.
\end{equation}
Now we know that
\begin{equation}
(\mathcal{L}_{V}R)(X,\xi)\xi=\mathcal{L}_{V}R(X,\xi)\xi-R(\mathcal{L}_{V}X,\xi)\xi-R(X,\mathcal{L}_{V}\xi)\xi-R(X,\xi)\mathcal{L}_{V}\xi.
\end{equation}
Using $(2.5)$ and $(3.18)-(3.22)$ in the previous equation we obtain
\begin{equation}
(\mathcal{L}_{V}R)(X,\xi)\xi=[2\lambda+2\mu-4n-(p+\frac{2}{2n+1})][X-\eta(X)\xi].
\end{equation}
Combining the equations $(3.15)$ and $(3.24)$ we get
\begin{equation}
[2\lambda+2\mu-4n-(p+\frac{2}{2n+1})][X-\eta(X)\xi]=0.\nonumber
\end{equation}
Tracing out the previous equation we finally obtain
\begin{equation}
\lambda+\mu=2n+(\frac{p}{2}+\frac{1}{2n+1})
\end{equation}
In view of $(3.25)$ and the sign of lambda we can conclude the required condition. This completes the proof.
\end{proof}
Now we are going establish an important lemma which will be required later in this section.

\begin{lemma}
  If $(g,V,\lambda,\mu)$ is a conformal $\eta$-Ricci soliton on a $(2n+1)$-dimensional Kenmotsu manifold $(M,g,\phi,\xi,\eta)$, then the Ricci tensor satisfies
  \begin{equation}
  (\mathcal{L}_{V}S)(Y,\xi)=(\xi r)\eta(Y)-(Yr).
  \end{equation}
\end{lemma}

\begin{proof}
Recalling the equation $(3.14)$ and taking inner product with respect to an arbitrary vector field $Z$ we can write
\begin{eqnarray}
g((\mathcal{L}_{V}R)(X,Y)\xi,Z) &=& 2[g((\nabla_XQ)Y,Z)-g((\nabla_YQ)X,Z)]\nonumber\\
&&+2[\eta(X)S(Y,Z)-\eta(Y)S(X,Z)]\nonumber\\
&&+4n[\eta(X)g(Y,Z)-\eta(Y)g(X,Z)].
\end{eqnarray}
Let us consider a local orthonormal basis $\{e_i: 1\leq i\leq (2n+1)\}$ of the manifold $M$. Setting $X=Z=e_i$ in $(3.27)$ and summing over the basis for $1\leq i\leq (2n+1)$ and then making use of the well-known formulas div$Q=\frac{1}{2}$grad$r$ and trace$\nabla Q=$grad$r$ we obtain
\begin{equation}
  (\mathcal{L}_{V}S)(Y,\xi)=-2[r+2n(2n+1)]\eta(Y)-(Yr).
  \end{equation}
  Again contracting $(3.15)$ we can easily get
  \begin{equation}
(\mathcal{L}_{V}S)(\xi,\xi)=0.
  \end{equation}
  Taking $Y=\xi$ in $(3.28)$ and making use of $(3.29)$ yields
  \begin{equation}
(\xi r)=-2n[r+2n(2n+1)].
  \end{equation}
 Hence in view of $(3.30)$ and $(3.28)$ leads to $(3.26)$ and this completes the proof.
\end{proof}
Next we will focus on $\eta$-Einstein Kenmotsu manifolds admitting conformal $\eta$-Ricci soliton and try to characterize them. Before that we need the following lemma from Venkatesha et. al., (for proof see \cite{Venk})

\begin{lemma}
  The scalar curvature $r$ of an $\eta$-Einstein Kenmotsu manifold satisfies the relation
  \begin{equation}
  Dr=(\xi r)\xi.
  \end{equation}
\end{lemma}
Now we are going to prove our main result of this theorem which completely characterizes the nature of the conformal $\eta$-Ricci soliton on $\eta$-Einstein Kenmotsu manifold.
 \begin{theorem}
  Let $(M,g,\phi,\xi,\eta)$ be a $(2n+1)$-dimensional $\eta$-Einstein Kenmotsu manifold. If $(g,V,\lambda,\mu)$ is a conformal $\eta$-Ricci soliton on $M$, then $M$ is an Einstein manifold.
 \end{theorem}

\begin{proof}
 Recalling $(3.31)$ and using it in $(3.26)$ yields
 \begin{equation}
 (\mathcal{L}_{V}S)(Y,\xi)=0.
  \end{equation}
 From $(2.10)$ we know that in a Kenmotsu manifold $S(X,\xi)=-2n\eta(X)$. Taking Lie derivative of this relation along an arbitrary vector field $Y$ and making use of $(3.32)$ we obtain
  \begin{equation}
 S(\mathcal{L}_{V}X,\xi)+S(X,\mathcal{L}_{V}\xi)=-2n[(\mathcal{L}_{V}\eta)(X)+\eta(\mathcal{L}_{V}X)].
  \end{equation}
  Again using $(3.25)$ in $(3.16)$ we get
  \begin{equation}
 (\mathcal{L}_{V}g)(X,\xi)=0.
  \end{equation}
  Now since the underlying manifold is an $\eta$-Einstein manifold it satisfies $(2.11)$ and combining it with $(2.10)$ gives $\alpha +\beta=-2n$. Contracting $(2.11)$ and in view of the foregoing relation we obtain
  \begin{equation}
 S(X,Y)=(1+\frac{r}{2n})g(X,Y)-(2n+1+\frac{r}{2n})\eta(X)\eta(Y).
  \end{equation}
  Making use of equations $(3.34)$ and $(3.35)$ in the equation $(3.33)$ we arrive at
  \begin{equation}
 (2n+1+\frac{r}{2n})\mathcal{L}_{V}\xi=0.
  \end{equation}
  This essentially implies that either $(2n+1+\frac{r}{2n})=0$ or $\mathcal{L}_{V}\xi=0$. Hence we have the following two cases.
  \par
  \medskip
  \textbf{Case-I:} If possible let $\mathcal{L}_{V}\xi=0$. Then it follows from $(3.36)$ that, $(2n+1+\frac{r}{2n})\neq0$ in some open set $\mathfrak{D}$ of the manifold $M$. Therefore using $\mathcal{L}_{V}\xi=0$ and $(2.5)$ we obtain
  \begin{equation}
 \nabla_\xi V=V-\eta(V)\xi.
  \end{equation}
 In view of $(3.37)$, the equation $(3.34)$ implies that
 \begin{equation}
 g(\nabla_XV,\xi)=-g(\nabla_\xi V,X)=-g(X,V)+\eta(X)\eta(V).
  \end{equation}
  We now recall the formula from Yano \cite{Ya}
  \begin{equation}
 (\mathcal{L}_{V}\nabla)(X,Y)=\nabla_X\nabla_Y-\nabla_{\nabla_XY}V+R(V,X)Y.
  \end{equation}
  Substituting $Y=\xi$ in the previous equation and making use of $(2.5)$, $(2.7)$, $(3.37)$ and $(3.38)$ we get $(\xi r)=0$. Thus from equation $(3.30)$ implies $(2n+1+\frac{r}{2n})=0$. This leads to a contradiction to the fact that $(2n+1+\frac{r}{2n})\neq0$ in some open set $\mathfrak{D}$ of the manifold $M$.
   \par
  \medskip
  \textbf{Case-II:} Let $(2n+1+\frac{r}{2n})=0$. Then by virtue of $(3.35)$ it immediately follows $S=(\frac{r}{2n}+1)g$ and hence the manifold $M$ is Einstein. This completes the proof.
\end{proof}

\medskip
\medskip
\section{Gradient conformal $\eta$-Ricci solitons on Kenmotsu manifolds}
$~~~$This section is devoted to the study of Kenmotsu manifolds admitting gradient conformal $\eta$-Ricci solitons and we try to characterize the potential vector field of the soliton. First, we prove the following important result.
\begin{lemma}
  If $(g,V,\lambda,\mu)$ is a gradient conformal $\eta$-Ricci soliton on a $(2n+1)$-dimensional Kenmotsu manifold $(M,g,\phi,\xi,\eta)$, then the Riemannian curvature tensor $R$ satisfies
  \begin{equation}
 R(X,Y)Df=[(\nabla_YQ)X-(\nabla_XQ)Y]+\mu[\eta(X)Y-\eta(Y)X].
  \end{equation}
\end{lemma}
\begin{proof}
  Since the data $(g,V,\lambda,\mu)$ is a gradient conformal $\eta$-Ricci soliton on the $(2n+1)$-dimensional Kenmotsu manifold, from the defining equation $(1.1)$ we get
  \begin{equation}
Hessf(X,Y)+S(X,Y)+[\lambda -(\frac{p}{2}+\frac{1}{n})]g(X,Y)+\mu \eta(X)\eta(Y)=0,
\end{equation}
  for all vector fields $X$ and $Y$ on $M$.\\
  Now the foregoing equation can be rewritten as
  \begin{equation}
\nabla_XDf=-QX-[\lambda -(\frac{p}{2}+\frac{1}{2n+1})]X-\mu \eta(X)\xi.
\end{equation}
Covariantly diffrentiating the previous equation along an arbitrary vector field $Y$ and using $(2.5)$ we obtain
 \begin{eqnarray}
\nabla_Y\nabla_XDf &=& -\nabla_Y(QX)-[\lambda -(\frac{p}{2}+\frac{1}{2n+1})]\nabla_YX\nonumber\\
&&-\mu[(\nabla_Y\eta(X))\xi+(Y-\eta(Y)\xi)\eta(X)].
\end{eqnarray}
Interchanging $X$ and $Y$ in $(4.4)$ gives
\begin{eqnarray}
\nabla_X\nabla_YDf &=& -\nabla_X(QY)-[\lambda -(\frac{p}{2}+\frac{1}{2n+1})]\nabla_XY\nonumber\\
&&-\mu[(\nabla_X\eta(Y))\xi+(X-\eta(X)\xi)\eta(Y)].
\end{eqnarray}
Again in view of $(4.3)$ we can write
\begin{eqnarray}
\nabla_{[X,Y]}Df &=& -Q(\nabla_XY-\nabla_YX)-\mu\eta(\nabla_XY-\nabla_YX)\xi\nonumber\\
&&-[\lambda -(\frac{p}{2}+\frac{1}{2n+1})](\nabla_XY-\nabla_YX).
\end{eqnarray}
Therefore substituting the values from $(4.4)$, $(4.5)$ and $(4.4)$ in the following well-known Riemannian curvature formula
\begin{equation}
R(X,Y)Z=\nabla_X\nabla_YZ-\nabla_Y\nabla_XZ-\nabla_{[X,Y]}Z,\nonumber
\end{equation}
we obtain our required expression $(4.1)$. This completes the proof.
\end{proof}
Now we proceed to prove our main result of this section.

\begin{theorem}
  Let $(g,V,\lambda,\mu)$ be a gradient conformal $\eta$-Ricci soliton on a $(2n+1)$-dimensional Kenmotsu manifold $(M,g,\phi,\xi,\eta)$. Then the potential vector field $V$ is pointwise collinear with the characteristic vector field $\xi$.
\end{theorem}

\begin{proof}
  Let us first recall the equation $(2.7)$ and taking its inner product with $Df$ yields
  \begin{equation}
  g(R(X,Y)\xi,Df)=(Yf)\eta(X)-(Xf)\eta(Y).\nonumber
  \end{equation}
  Again we know that $g(R(X,Y)\xi,Df)=-g(R(X,Y)Df,\xi)$ and in view of this the previous equation becomes
   \begin{equation}
 g(R(X,Y)Df,\xi)=(Xf)\eta(Y)-(Yf)\eta(X).
  \end{equation}
  Now taking inner product og $(4.1)$ with $\xi$ and using $(3.1)$ we obtain
  \begin{equation}
 g(R(X,Y)Df,\xi)=0.
  \end{equation}
  Thus combining $(4.7)$ and $(4.8)$ we arrive at
  \begin{equation}
 (Xf)\eta(Y)=(Yf)\eta(X).\nonumber
  \end{equation}
  Taking $Y=\xi$ in the foregoing equation gives us $(Xf)=(\xi f)\eta(X)$, which essentially implies
  \begin{equation}
 g(X,Df)=g(X,(\xi f)\xi),
  \end{equation}
  for all smooth vector fields $X$ on $M$. Since the equation $(4.9)$ is true for all $X$, we can conclude that $V=Df=(\xi f)\xi$. This completes the proof.
\end{proof}

\medskip
\medskip
\section{conformal $\eta$-Ricci solitons on Kenmotsu manifolds under generalized $D$-conformal deformations}
$~~~$Given an almost contact metric structure $(M,g,\phi,\xi,\eta)$, the generalized $D$-conformal deformation \cite{AleCar} on $M$ is given by
\begin{equation}
\phi^*=\phi,~~~\xi^*=\frac{1}{a}\xi,~~~\eta^*=a\eta,~~~g^*=bg+(a^2-b)\eta\otimes\eta,
\end{equation}
where $a$ and $b$ are two positive functions on $M$. Then it is easy to see that $(M,g^*,\phi^*,\xi^*,\eta^*)$ is also an almost contact metric manifold \cite{AleCar}. We notice that, the transformation $(5.1)$ reduces to
\begin{enumerate}
  \item $D$-homothetic deformation \cite{STan} for $a=b=constant$ and
  \item conformal deformation \cite{RSAG} for $a^2=b$.
\end{enumerate}
\par
\medskip
According to \cite{RSAG}, If $(M,g^*,\phi^*,\xi^*,\eta^*)$ is an almost contact metric manifold obtained by the generalized $D$-conformal deformation $(5.1)$, of a Kenmotsu manifold $(M,g,\phi,\xi,\eta)$, then $(M,g^*,\phi^*,\xi^*,\eta^*)$ is also a Kenmotsu manifold and $a$ and $b$ are constant along the direction of the characteristic vector field $\xi$. In what follows, the quantities in $(M,g^*,\phi^*,\xi^*,\eta^*)$ are denoted with $*$ and the quantities in $(M,g,\phi,\xi,\eta)$ are denoted without $*$ respectively.
\par
\medskip
By \cite{Naga}, the Levi-Civita connections $\nabla^*$ and $\nabla$ are related by
\begin{equation}
{\nabla^*}_XY=\nabla_XY+\frac{a^2-b}{a^2}g(\phi X,\phi Y)\xi,
\end{equation}
for all vector fields $X$ and $Y$ on $M$. Making use of $(5.2)$, the Riemannian curvature tensor $R^*$ of $(M,g^*,\phi^*,\xi^*,\eta^*)$ can be calculated as follows
\begin{eqnarray}
  R^*(X,Y)Z &=& R(X,Y)Z+\frac{a^2-b}{a^2}[g(\phi Y,\phi Z)X-g(\phi X,\phi Z)Y] \nonumber\\
   && +g(\phi Y,\phi Z)[\frac{2b}{a^3}(Xa)-\frac{1}{a^2}(Xb)]+g(\phi X,\phi Z)[\frac{2b}{a^3}(Ya)-\frac{1}{a^2}(Yb)],
\end{eqnarray}
for all vector fields $X,Y,Z$ on $M$. On contracting $(5.3)$, obtain the Ricci tensor $S^*$ of the generalized $D$-conformally deformed Kenmotsu manifold as
\begin{equation}
S^*(X,Y)=S(X,Y)+\frac{2n(a^2-b)}{a^2}[g(X,Y)-\eta(X)\eta(Y)].
\end{equation}
\par
\medskip
Now we are in a position to prove our main result of this section, on the invariance of conformal $\eta$-Ricci soliton on Kenmotsu manifolds under generalized $D$-conformal deformation.
\begin{theorem}
  Under generalized $D$-conformal deformation of a Kenmotsu manifold $(M,g,\phi,\xi,\eta)$, an $\eta$-Einstein conformal $\eta$-Ricci soliton remains invariant.
\end{theorem}
\begin{proof}
Let the data $(g,V,\lambda,\mu)$ be a conformal $\eta$-Ricci soliton on a Kenmotsu manifold $(M,g,\phi,\xi,\eta)$, then from $(3.3)$ using $Y=\xi$ and recalling $(2.10)$ gives
\begin{equation}
(\mathcal{L}_{V}g)(X,\xi)-4n\eta(X)+[2\lambda +2\mu-(p+\frac{2}{2n+1})]\eta(X)=0.
\end{equation}
 Now, taking Lie derivative of $\eta(\xi)=1$ along an arbitrary vector field $V$ yields
  \begin{equation}
(\mathcal{L}_{V}\eta)(\xi)+\eta(\mathcal{L}_{V}\xi)=0.
\end{equation}
Again using the fact $(\mathcal{L}_{V}g)(X,\xi)=(\mathcal{L}_{V}\eta)(X)-g(X,\mathcal{L}_{V}\xi)$ in $(5.5)$, then putting $X=\xi$ and in view of $(5.6)$ we obtain
\begin{equation}
\eta(\mathcal{L}_{V}\xi)=[(\lambda +\mu-n)-(\frac{p}{2}+\frac{1}{2n+1})].
\end{equation}
Recalling $\mathcal{L}_{V}\xi=\eta(\mathcal{L}_{V}\xi)\xi$ from \cite{RSAG}, and using $(5.7)$ we get
\begin{equation}
g(X,\mathcal{L}_{V}\xi)=[(\lambda +\mu-n)-(\frac{p}{2}+\frac{1}{2n+1})]\eta(X).
\end{equation}
In view of $(5.8)$ and $(\mathcal{L}_{V}g)(X,\xi)=(\mathcal{L}_{V}\eta)(X)-g(X,\mathcal{L}_{V}\xi)$, the equation $(5.5)$ reduces to
\begin{equation}
(\mathcal{L}_{V}\eta)(X)=[(n-\lambda -\mu)+(\frac{p}{2}+\frac{1}{2n+1})]\eta(X).
\end{equation}
Now, Lie differentiating $g^*=bg+(a^2-b)\eta\otimes\eta$ along the vector field $V$ and making use of $(5.1)$ we obtain
\begin{eqnarray}
  (\mathcal{L}_{V}g^*)(X,Y) &=& (Vb)g(X,Y)+b(\mathcal{L}_{V}g)(X,Y)+[2a(Va)-(Vb)]\eta(X)\eta(Y) \nonumber\\
   && +(a^2-b)[(\mathcal{L}_{V}\eta)(X)\eta(Y)+\eta(X)(\mathcal{L}_{V}\eta)(Y)],\nonumber
\end{eqnarray}
for all vector fields $X,Y$ on $M$. Using $(5.9)$ the previous equation becomes
\begin{eqnarray}
  (\mathcal{L}_{V}g^*)(X,Y) &=& (Vb)g(X,Y)+b(\mathcal{L}_{V}g)(X,Y)+[2a(Va)-(Vb)]\eta(X)\eta(Y) \nonumber\\
   && +2(a^2-b)[(n-\lambda -\mu)+(\frac{p}{2}+\frac{1}{2n+1})]\eta(X)\eta)(Y).
\end{eqnarray}
Now making use of $(5.1)$, $(5.4)$ and $(5.10)$, we can compute the following expression
\begin{eqnarray}
 && (\mathcal{L}_{V}g^*)(X,Y)+2S^*(X,Y)+[2\lambda -(p+\frac{2}{2n+1})]g^*(X,Y)+2\mu \eta^*(X)\eta^*(Y) \nonumber\\
 &=& 2(1-b)S(X,Y)+[(Vb)+\frac{4n(a^2-b)}{a^2}]g(X,Y) \nonumber\\
 && +[2a(Va)-(Vb)+\frac{2n(a^2-b)(a^2-2)}{a^2}]\eta(X)\eta(Y).
\end{eqnarray}
Hence from $(5.11)$ we can conclude that the data $(g^*,V,\lambda,\mu)$ is also a conformal $\eta$-Ricci soliton iff the right hand side vanishes, i.e.; iff
\begin{equation}
  S(X,Y)=\alpha g(X,Y)+\beta \eta(X)\eta(Y),
  \end{equation}
where
\begin{eqnarray}
  \alpha &=& \frac{1}{2(b-1)}[(Vb)+\frac{4n(a^2-b)}{a^2}] \nonumber\\
  \beta &=& \frac{1}{2(b-1)}[2a(Va)-(Vb)+\frac{2n(a^2-b)(a^2-2)}{a^2}].\nonumber
\end{eqnarray}
Now, substituting $(5.12)$ in $(5.4)$ and making use of $(5.1)$ we obtain
\begin{eqnarray}
 S^*(X,Y) &=& \frac{1}{2(b-1)}[\frac{(Vb)}{b}+\frac{2n(a^2-b)}{a^2}]g^*(X,Y) \nonumber\\
   && +\frac{1}{2(b-1)}[\frac{2(Va)}{a}-\frac{(Vb)}{b}+\frac{2n(a^2-b)(2a^2-b)}{a^4}]\eta^*(X)\eta^*(Y).
\end{eqnarray}
Therefore $(M,g^*)$ is also an $\eta$-Einstein manifold. This completes the proof.
\end{proof}
Now, if we put $a=b=constant$ in $(5.13)$, we get\\
 $S^*(X,Y)=\frac{2n}{a}g^*(X,Y)-\frac{2n(1-2a)}{a^2}\eta^*(X)\eta^*(Y)$. This leads to the following
 \begin{corollary}
    Under $D$-homothetic deformation of a Kenmotsu manifold $(M,g,\phi,\xi,\eta)$, an $\eta$-Einstein conformal $\eta$-Ricci soliton remains invariant.
 \end{corollary}
 Again, for $a^2=b$, equation $(5.13)$ reduces to $S^*(X,Y)=\frac{(Va)}{a(a^2-1)}g(X,Y)$. Hence we can conclude that
 \begin{corollary}
   Under a conformally deformed Kenmotsu manifold $(M,g,\phi,\xi,\eta)$, an $\eta$-Einstein conformal $\eta$-Ricci soliton deforms to an Einstein metric.
 \end{corollary}

\medskip
\medskip
\section{Example of a conformal $\eta$-Ricci soliton on a $5$-dimensional Kenmotsu manifold}
$~~~$Let us consider the $5$-dimensional manifold $M=\{(u_1,u_2,v_1,v_2,w)\in\mathbb{R}^5:w\neq 0\}$. Define a set of vector fields
$\{e_i: 1\leq i\leq 5\}$ on the manifold $M$ given by
\begin{equation}
e_1=w\frac{\partial}{\partial u_1}, ~~e_2=w\frac{\partial}{\partial u_2}, ~~e_3=w\frac{\partial}{\partial v_1}, ~~e_4=w\frac{\partial}{\partial v_2}, ~~e_5=-w\frac{\partial}{\partial w}.\nonumber
\end{equation}
Let us define the indefinite metric $g$ on $M$ by
\begin{equation}
g(e_i,e_j)=\left\{ \begin{array}{rcl}
1, & \mbox{for}
& i=j \\ 0, & \mbox{for} & i\neq j
\end{array}\right.\nonumber
\end{equation}1
for all $i,j=1,2,3,4,5$. Now considering $e_5=\xi$, let us take the $1$-form $\eta$, on the manifold $M$, defined by
\begin{equation}
\eta(U)=g(U,e_5)=g(U,\xi),~~~\forall U\in TM.\nonumber
\end{equation}
Then it can be observed that $\eta(e_5=1)$. Let us define the $(1,1)$ tensor field $\phi$ on $M$ as
\begin{equation}
\phi(e_1)=e_2, ~~\phi(e_2)=-e_1,~~\phi(e_3)=e_4,~~\phi(e_4)=-e_3,~~\phi(e_5)=0.\nonumber
\end{equation}
Then using the linearity of $g$ and $\phi$ it can be easily checked that
\begin{equation}
\phi^2(U)=-U+\eta(U)\xi,~~g(\phi U,\phi V)=g(U,V)-\eta(U)\eta(V),~~~\forall U,V\in TM.\nonumber
\end{equation}
Hence the structure $(g,\phi,\xi,\eta)$ defines an almost contact metric structure on the manifold $M$.\\
Now, using the definitions of Lie bracket, direct computations give us \\
$[e_i,e_5]=e_i;~~\forall i=1,2,3,4,5$ and all other $[e_i,e_j]$ vanishes.
Again the Riemannian connection $\nabla$ of the metric $g$ is defined by the well-known Koszul's formula which is given by
\begin{eqnarray}
  2g(\nabla_XY,Z) &=& Xg(Y,Z)+Yg(Z,X)-Zg(X,Y) \nonumber\\
   && -g(X,[Y,Z])+g(Y,[Z,X])+g(Z,[X,Y]).\nonumber
\end{eqnarray}
Using the above formula one can easily calculate that \\
$\nabla_{e_i}e_i=- e_5$, $\nabla_{e_i}e_5=-e_i$; for i=1,2,3,4 and all other $\nabla_{e_i}e_j$ vanishes. Thus it follows that $\nabla_X\xi=(X-\eta(X)\xi),~~~\forall X\in TM$. Therefore the manifold $(M,g)$ is a $5$-dimensional Kenmotsu manifold.\\
Now using the well-known formula $R(X,Y)Z=\nabla_X\nabla_YZ-\nabla_Y\nabla_XZ-\nabla_{[X,Y]}Z$ the non-vanishing components of the Riemannian curvature tensor $R$ can be easily obtained as
\begin{multline}
  $~~~~~~~~~~~~~$R(e_1,e_2)e_2=R(e_1,e_3)e_3=R(e_1,e_4)e_4=R(e_1,e_5)e_5=-e_1, \\
  R(e_1,e_2)e_1=e_2,$~~$R(e_1,e_3)e_1=R(e_1,e_3)e_2=R(e_1,e_3)e_5=e_3, \\
  $~~~~~$R(e_1,e_2)e_3=R(e_1,e_2)e_4=R(e_1,e_2)e_5=-e_2,$~~$R(e_1,e_2)e_4=-e_3, \\
  R(e_1,e_2)e_2=R(e_1,e_2)e_1=R(e_1,e_2)e_4=R(e_1,e_2)e_3=e_5,$~~~~$ \\
  R(e_1,e_2)e_1=R(e_1,e_2)e_2=R(e_1,e_2)e_3=R(e_1,e_2)e_5=e_4.$~~~~~~~~~~~~~~~~$ \nonumber
\end{multline}
From the above values of the curvature tensor, we obtain the components of the Ricci tensor as follows
\begin{equation}
S(e_1,e_1)=S(e_2,e_2)=S(e_3,e_3)=S(e_4,e_4)=S(e_5,e_5)=-4.
\end{equation}
Therefore using the above values in the definition of conformal $\eta$-Ricci soliton, we can calculate $\lambda=3+(\frac{p}{2}+\frac{1}{5})$ and $\mu=1$. Hence we can say that for $\lambda=3+(\frac{p}{2}+\frac{1}{5})$ and $\mu=1$, the data $(g,\xi,\lambda,\mu)$ defines a conformal $\eta$-Ricci soliton on the $5$-dimensional Kenmotsu manifold $(M,g,\phi,\xi,\eta)$.

\medskip
\medskip
\section{Conclusion}
$~~~$In this article, we have investigated the effects of conformal $\eta$-Ricci soliton and gradient conformal $\eta$-Ricci soliton on the Kenmotsu metric. More precisely, we have identified the Einstein metric from a large class of metrics conformal $\eta$-Ricci soliton on Kenmotsu manifolds. It is well-known that Einstein manifolds are very special class of manifolds having extensive relevance in the theory of general relativity and hence our results will play an important role in this field of research. Furthermore, it is interesting to investigate conformal $\eta$-Ricci solitons on other contact metric manifolds and there is further scope of research in this direction within the framework of various complex manifolds like K\"{a}hler manifold, Hopf manifold etc. Finally, for future research, we conclude this article by posing a question:
\par
\medskip
\textbf{Question:} Is theorem $3.5$ still true if the $\eta$-Einstein condition is dropped?
\par
\medskip
\medskip

\bebi
\bibitem{AleCar}P. Alegre, A. Carriazo ~{\em Generalized Sasakian space form and conformal changes of the metric},
~Results in Mathematics, ~59,~pp.485-493, ~(2011).
\bibitem{NBa} N. Basu, A. Bhattacharyya, ~{\em Conformal Ricci soliton in Kenmotsu manifold},
~Glob. J. of Adv. Res. on Clas. and Mod. Geom., ~4(1),~pp. 15-21, ~(2015).
\bibitem{Bla} D. E. Blair,~{\em Riemannian Geometry of Contact and Symplectic Manifolds},
~Birkhauser,~Second Edition,~(2010).
\bibitem{CaCra} C. Calin, M. Crasmareanu, ~{\em $\eta$-Ricci solitons on Hopf hypersurfaces in complex space forms},
~Revue Roumaine de Math. Pures et Appl.,~57(1),~pp.53-63,~(2012).
\bibitem{ChoK} J. T. Cho, M. Kimura, ~{\em Ricci solitons and real hypersurfaces in a complex space forms},
~Tohoku Math. J.,~61,~pp.205-212,~(2009).
\bibitem{Fis}A. E. Fischer, ~{\em An Introduction to Conformal Ricci flow},
~Classical and Quantum Gravity, ~Vol. 21, ~Issue 3, ~pp. S171-S218, ~(2004).
\bibitem{H1}R. S. Hamilton, ~{\em Three manifolds with positive Ricci curvature},
~Journal of Differential Geometry 17, ~255-306, ~(1982).
\bibitem{H2}R. S. Hamilton, ~{\em The formation of singularities in the Ricci flow},
~Surveys in Differential Geometry, ~Vol. II, ~(Cambridge, MA, 1993), ~pp. 7-136, ~Int. Press, ~Cambridge, ~MA, ~(1995).
\bibitem{DG2} D. Ganguly, A. Bhattacharyya, ~{\em A study on conformal Ricci solitons in the framework of $(LCS)_n$-manifolds},
~Ganita,~vol-70(2),~pp-201-216,~(2020).
\bibitem{DG3}D. Ganguly, S. Dey, A. Ali, A. Bhattacharyya,,~{\em Conformal Ricci soliton and quasi-Yamabe soliton on generalized Sasakian space form},
~Journal of Geometry and Physics,~vol-169,~104339, ~(2021).
\bibitem{DGPHui} D. G. Prakasha, S. K. Hui, K. Vikas,~{\em On weakly $\phi$-Ricci symmetric Kenmotsu manifolds},
~Int. J. Pure and Appl. Math.,~95(4),515-521~(2014).
\bibitem{AG} R. Sharma, A. Ghosh,~{\em Kenmotsu $3$-metric as a Ricci soliton},
~Chaos Solitons Fractals,~44,~pp.647-650,~(2011).
\bibitem{KK}K. Kenmotsu,~{\em A class of almost contact Riemannian manifold},
~Tohoku Math. J.,~24,~pp.93-103,~(1972).
\bibitem{Naga}H. G. Nagaraja, D. L. Kiran Kumar,~{\em Ricci solitons in Kenmotsu manifold under generalized $D$-conformal deformation},
~Lobachevskii J. of Math.~40(2),~pp.195-200,~(2019).
\bibitem{AASkhKKBais}A. A. Shaikh, K. K. Baishya, S. Eyasmin,~{\em $D$-homothetic deformation of trans-Sasakian structure},
~Demonstr. Math.,~41,~pp.171-188,~(2008).
\bibitem{RSAG} R. Sharma, A. Ghosh,~{\em Sasakian $3$-manifold as a Ricci soliton represents the Heisenberg group},
~Int. J. of Geom. Meth. in Mod. Phys.~8,~pp.149-154,~(2011).
\bibitem{MDS1}M. D. Siddiqi,~{\em Conformal $\eta$-Ricci solitons in $\delta$-Lorentzian trans Sasakian manifolds},
~Int. J. of Maps in Math.,~1(1),~pp.15-34,~(2018).
\bibitem{SuguNaka}S. Suguri, S. Nakayama,~{\em $D$-conformal deformations on almost contact metric structure},
~Tensor(N.S.),~28,~pp.125-129,~(1974).
\bibitem{STan}S. Tanno,~{\em The topology of contact Riemannian manifolds},
~Illinois J. of Math.~12,~pp.700-717,~(1968).
\bibitem{Venk}V. Venkatesha, D. M. Naik, H. A. Kumara~{\em *-Ricci solitons and gradient almost *-Ricci solitons on Kenmotsu manifolds},
~Mathematica Slovaca,~69,~pp.1447-1458,~(2019).
\bibitem{Ya}K. Yano,~{\em Integral formulas in Riemannian geometry},~New York,~Marcel Dekker,~(1970).
\eebi

\aua
    Dipen Ganguly\\
    Department~of~Mathematics, Jadavpur~University\\
    Kolkata-700032, West Bengal,~India.\\
    E-mail: dipenganguly1@gmail.com\\

\end{document}